\documentclass[12pt]{amsart}
\usepackage{amsmath,amssymb,amsthm}
\usepackage{hyperref}
\usepackage{latexsym}            
\usepackage{color,colordvi,graphicx}

\numberwithin{equation}{section}

\newtheorem{theorem}{Theorem}[section]
\newtheorem{proposition}[theorem]{Proposition}
\newtheorem{lemma}[theorem]{Lemma}
\newtheorem{corollary}[theorem]{Corollary}

\theoremstyle{remark}

\newtheorem{remark}[theorem]{Remark}

\newcounter{FNC}[page]
\def\fauxfootnote#1{{\addtocounter{FNC}{2}\Magenta{$^\fnsymbol{FNC}$}%
     \let\thefootnote\relax\footnotetext{\Magenta{$^\fnsymbol{FNC}$#1}}}}
\newcommand{\defcolor}[1]{{\color{blue}#1}}
\newcommand{\demph}[1]{\defcolor{{\sl #1}}}

\newcommand{\GL}{\mbox{\rm GL}}
\newcommand{\Gr}{\mbox{\rm Gr}}
\newcommand{\Fl}{\mbox{\rm Fl}(1,5;V)}
\newcommand{\Wr}{\mbox{\rm Wr}}

\newcommand{\calS}{\mathcal S}
\newcommand{\bwedge}{{\textstyle \bigwedge}}

\newcommand{\CC}{{\mathbb C}}
\newcommand{\PP}{{\mathbb P}}
\newcommand{\RR}{{\mathbb R}}

\DeclareMathOperator{\Mat}{{\rm Mat}}

\title[Nontrivial linear projections on the Grassmannian $\Gr_3(\CC^6)$]{Nontrivial linear projections\\
  on the Grassmannian $\Gr_3(\CC^6)$}
\author{Yanhe Huang}
\address{Yanhe Huang \\
         Department of Mathematics\\
         University of California\\ Berkeley\\
         970 Evans Hall\\
         Berkeley, CA 94720-3840
         USA}
\email{yanhe\_huang@berkeley.edu}
\author{George Petroulakis}
\address{George Petroulakis \\
	Athens, Greece}
\email{Georgios.Petroulakis.1@city.ac.uk}

\author{Frank Sottile}
\address{Frank Sottile \\
         Department of Mathematics\\
         Texas A\&M University\\
         College Station\\
         Texas \ 77843\\
         USA}
\email{sottile@math.tamu.edu}
\urladdr{\url{http://www.math.tamu.edu/~sottile}}
\author{Igor Zelenko}
\address{Igor Zelenko\\
         Department of Mathematics\\
         Texas A\&M University\\
         College Station\\
         Texas \ 77843\\
         USA}
\email{zelenko@math.tamu.edu}
\urladdr{\url{http://www.math.tamu.edu/~zelenko}}

\thanks{Sottile was supported in part by Simons Foundation Collaboration Grant for Mathematicians 636314.}
\thanks{Zelenko was partly supported by NSF grant DMS-1406193 and Simons Foundation Collaboration Grant for Mathematicians
  524213.} 
\keywords{Wronski map, Pl\"{u}cker embedding, 3-forms in $\CC^6$, self-adjoint linear ordinary differential operators,
  symmetric linear control systems, pole placement map} 
\subjclass[2010]{14M15, 34A30, 93B55}

\begin{document}

\begin{abstract}
A typical linear projection of the Grassmannian in its Pl\"ucker embedding is injective, unless its image
is a projective space.
A notable exception are self-adjoint linear projections, which have even degree.
We consider linear projections of $\Gr_3\CC^6$ with low-dimensional centers of projection.
When the center has dimension less than five, we show that the projection has degree 1.
When the center has dimension five and the projection has degree greater than 1, we show that it is self-adjoint.
\end{abstract}
\maketitle

\section{Introduction}
Consider a linear  ordinary differential operator (ODO) of order $n$
 \begin{equation}
 \label{ode}
   L\,x(t)\  =\   x^{(n)}(t)+a_{n-1}(t)x^{(n-1)}(t)+\cdots+a_{0}(t)x(t)\,,
 \end{equation}
where $a_0,\dotsc,a_{n-1}$ are complex-valued  continuous functions on an interval $I\subset\RR$.
Let \defcolor{$V_L$} be the space of complex-valued solutions of the homogeneous 
equation $Lx=0$.

The \demph{Wronskian} of $m$ smooth functions $f_1(t),\dotsc, f_m(t)$ on $I$ is the determinant 
\[
   \defcolor{\Wr}\bigl(f_1(t),f_2(t), \dotsc, f_m(t)\bigr)\ :=\
   \det \begin{pmatrix}f_1(t)&f_2(t)&\dotsb& f_m(t)\\
                      f_1'(t)&f'_2(t)&\dotsb &f_m'(t)\\
                      \vdots &\vdots&\ddots&\vdots\\
                     f_1^{(m-1)}(t)&f_2^{(m-1)}(t)&\dotsb&f_m^{(m-1)}(t)\end{pmatrix}\,.
\]
The Wronskian $\Wr\bigl(f_1(t),\ldots, f_m(t)\bigr)$ is not identically zero when $f_1(t),\dotsc,f_m(t)$ form a basis 
of an $m$-dimensional subspace $\Lambda$ in $V_L$.
If  $g_1(t),\dotsc,g_m(t)$ is another basis, then
\[
  \Wr\bigl(g_1(t),g_2(t), \ldots, g_m(t)\bigr)\ =\ 
  c\Wr\bigl(f_1(t),f_2(t), \ldots, f_m(t)\bigr)\,,
\]
where $c$ is the determinant of the transition matrix between the bases.
Therefore, the one-dimensional linear subspace of $C^\infty(I)$ spanned by the Wronskian
$\Wr\bigl(f_1(t),\ldots,f_m(t)\bigr)$ depends only upon $\Lambda$.
This element of the projective space $\PP C^\infty(I)$ is called the \demph{Wronskian of the subspace $\Lambda$}. 
This defines  the \demph{Wronski map $\Wr_{L,m}$} from the Grassmannian \defcolor{$\Gr_mV_L$} of $m$-dimensional
subspaces of $V_L$ to $\PP C^\infty(I)$. 

For  complex algebraic varieties $X,Y$ of the same dimension and a dominant map $F\colon X\to Y$,
the number of points in a preimage $F^{-1}(y)$ for $y\in Y$ is constant over an open dense subset of $Y$.
This constant number is the \demph{degree} of the map $F$~\cite{Harris}. 

Consider this for the Wronski map $\Wr_{L,m}$ when the image of $\Gr_mV_L$ has the same dimension as
$\Gr_mV_L$.
For generic linear ODO $L$ of order $n$ and any $m\in \{2,\ldots, n{-}1\}$ the Wronski map $\Wr_{L,m}$ is injective  (see
Remark \ref{genrem}) and so $\Wr_{L,m}$ has degree $1$. 
For any $L$, is it injective when $m=1$ or $m=n{-}1$.
We are interested in the following question.\medskip

\noindent{\bf Question 1.}
{\it Under what conditions on a linear ODO $L$ of order $n$  and on $1<m<n{-}1$ does the Wronski map  $\Wr_{L, m}$
  have degree greater than $1$?} 
\medskip

The classical Wronski map is when $V$ is the space of polynomials of degree $n{-}1$.
This corresponds to the ODO
$\defcolor{L_0}\, x(t)=x^{(n)}(t)$. 
Work of Schubert in 1886~\cite{Sch1886c}, combined with a result of Eisenbud and Harris in
1983~\cite {EH83} shows that the Wronski map $\Wr_{L_0, m}$ has degree 
 \begin{equation}
    \label{Schubert}
      \frac{1!2!\dotsb(n-m-1)!\cdot(m(n-m))!}{m! (m+1)!\dotsb(n-1)!}\,.
 \end{equation}
The degree exceeds $1$ except in the trivial cases of $m=1$ or $m=n{-}1$.

Three of us addressed Question 1 in a previous paper~\cite{HSZ2017}.
The operator
 \begin{equation}
 \label{adj}
   \defcolor{L^{\ast}}\,x(t)\ :=\ (-1)^{n} x^{(n)}(t)\ +\ \sum_{i=1}^{n-1}(-1)^{i}(a_{i}x)^{(i)}(t)
 \end{equation}
is \demph{(formally) adjoint} to the operator $L$~\eqref{ode}.
An operator $L$ is a \demph{(formally) self-adjoint differential operator} if $L^{\ast}=L$.
When $L$ is self-adjoint, its order $n$ is even.
 
Two linear ODOs $L$ and $\widetilde L$ on $I$ are \demph{equivalent} if
there exists a smooth nonvanishing function $\mu$ on $I$ such that
\[
 \widetilde L x\ =\ \frac{1}{\mu} L\bigl(\mu x\bigr)\,.
\]

We paraphrase two results from~\cite{HSZ2017}.
For both, $L$ is a linear ODO of order $n$.\medskip

\noindent{\bf Theorem 2.9 of~\cite{HSZ2017}}
{\it If $L$ is equivalent to a self-adjoint operator and $n=2m$, then the Wronski map $\Wr_{L, m}$ has even
  degree.}\medskip

\noindent{\bf Corollary 1.8 of~\cite{HSZ2017}}
{\it  If the Wronski map $\Wr_{L, m}$ has degree $2$, then $n=2m$ and $L$ is equivalent to a self-adjoint linear
  operator.}\medskip 

The proof of~\cite[Thm.\ 2.9]{HSZ2017} is based on two observations.
First, if $L$ is equivalent to a self-adjoint operator then the space $V_L$ is endowed with a canonical (up to a nonzero
scaling) symplectic structure $\sigma_L$.
Second, if $\Lambda^\angle$ is the skew-orthogonal complement of an $m$-dimensional subspace $\Lambda$ of $V_L$
with respect to the form $\sigma_L$, then 
\begin{equation}
\label{mainid}
\Wr_{L,m} (\Lambda^\angle)\ =\ \Wr_{L,m}(\Lambda),
\end{equation}
so that the Wronskian is preserved under taking skew-orthogonal complement.

From \eqref{Schubert} it follows that for the ODO $L_0\, x(t)=x^{(n)}(t)$ with  $n\geq 5$ and $m\notin\{1, n{-}1\}$ the
Wronski map $\Wr_{L_0, m}$ has degree  greater than $2$.
Thus  $n=2m$ is not necessary for the degree of the Wronski map to exceed 1.
\medskip

\noindent{\bf Question 2.}
{\it When $n=2m$, does the statement of\/~\cite[Cor.\ 1.8]{HSZ2017} generalize as follows:
If  the Wronski map $\Wr_{L, m}$ of a $2m$-th order linear ODO $L$ has degree greater than $2$, is $L$ equivalent to
a self-adjoint operator?}
\medskip

We address a generalization of Question 2.
The Grassmannian $\Gr_mV_L$ is a subvariety of Pl\"ucker space $\PP{\bwedge^m}{V_L}$.
Given a linear subspace $\PP Z\subset \PP{\bwedge^m}{V_L}$ ($Z$ is a linear subspace of $\bwedge^mV_L$), the linear
projection with center $\PP Z$ is the map $\PP{\bwedge^m}{V_L}\smallsetminus\PP Z \to \PP(\bwedge^mV_L)/Z$ induced by
the map $\bwedge^mV_L\to(\bwedge^mV_L)/Z$. 
When $\PP Z$ is disjoint from the Grassmannian, it induces the \demph{linear projection}
$\defcolor{\pi_Z}\colon\Gr_mV_L\to\PP(\bwedge^mV_L)/Z$.

Proposition 2.3 of~\cite{HSZ2017} identifies the Wronski map with a linear projection.
We explain that.
Given a basis  $f_1,\dotsc, f_n$ for $V_L$, let $f_1^*,\ldots f_n^*\in V_L^*$ be its dual basis and   set
\begin{equation}
\label{ceq}
  \defcolor{c(t)}\ :=\ \sum_{i=1}^n f_i(t) f_i^*\in V_L^*\,, \mbox{ for }t\in I\,.
\end{equation}
Fix $m\in\{1,\dotsc,n{-}1\}$ and define the following subspace of $\bwedge^m  V_L^*$,
 \begin{equation}
\label{defspaceK}
  \defcolor{X_L}\ :=\
  \bigl\langle c(t)\wedge c'(t)\wedge\cdots \wedge c^{(m-1)}(t) \mid t\in I\bigr\rangle\,,
\end{equation}
where 
\[
  c^{(j)}(t)\ =\ \sum_{i=1}^n f_i^{(j)}(t) f_i^*\in V_L^*\,.
\]
By~\cite[Prop.\ 2.3]{HSZ2017}, the Wronski map takes values in the space $X_L^*$ dual to $X_L$, which is
$(\bwedge^m V)/X_L^\perp$,  where
\[
  X_L^\perp\ =\ \{w\in \bwedge^m V_L\mid \omega(v)=0\,\  \forall v\in X_L\}
\]
is the annihilator of $X_L$ (for details see \cite[pp.\ 755-6]{HSZ2017}).

\begin{remark}
\label{genrem}
For generic linear ODO $L$, $X_L=\bwedge^m V_L^*$, which implies that the Wronski map is injective (this is a
consequence of Proposition~\ref{equiv} below, as $X_L^\perp=0$).  \hfill$\diamond$
\end{remark}

\begin{remark}
\label{genselfrem}
As a consequence of \cite[Sect.\ 2.3]{HSZ2017}, a linear ODO of order $2m$ is self-adjoint if and only if there exists a
symplectic form $\sigma$ on $V_L^*$ such that  
 \[
 X_L ^\perp\ \supseteq\ \CC \sigma\wedge \bwedge^{m-2} V_L\,.
 %
 %
 \]
Moreover, the canonical symplectic form on $V_L$ is induced by the form $\sigma$ through the identification of $V_L$ with
$V_L^*$ via $\sigma$. 
This inclusion implies that  
 \[
   \dim X_{L}^\perp\ \geq\ \dim \bwedge^{m-2}  V_L^*\ =\ \binom {2m}{m-2}\,,
 \]
with equality for a generic self-adjoint linear ODO of order $2m$.
When $m=3$ and $L$ is self-adjoint, the minimal possible dimension of $X_{L}^\perp$ is $6$.\hfill$\diamond$
\end{remark}

Let $V$ be an even-dimensional complex vector space and $1< m<\dim V$.
A linear subspace $Z\subset\bwedge^mV$ is \demph{self-adjoint} if there exists a
symplectic form $\sigma$ on $V^*$ such that 
%
 \[
  Z\ \supseteq\ \CC \sigma\wedge \bwedge^{m-2} V\,.
 \]
We state our main results.\medskip

\noindent{\bf Theorem~\ref{m=2iff}.} {\it
  When $m=2$ and $n=4$, if\/ $\PP Z$ is a linear subspace disjoint from $\Gr_2\CC^4$, then $Z$ is self-adjoint.
  }\medskip

When $m=3$ and $n=6$, we consider centers $Z$ of projective dimensions four or five.
  
\begin{proposition}
  \label{Prop:main}
  Suppose that $m=3$ and $n=6$.
  Let $Z\subset\bwedge^3\CC^6$ be a linear subspace with $\PP Z$ disjoint from $\Gr_3\CC^6$.
  \begin{enumerate}
    \item {\rm [Corollary \ref{wendy2}]} If $\dim\PP Z\leq 4$, then $\pi_Z$ has degree $1$.
    \item {\rm [Theorem \ref{mainm=3theor}]} If $\dim\PP Z=5$, then $\pi_Z$ has degree greater than $1$ if and only if $Z$
      is self-adjoint.
   \end{enumerate}
\end{proposition}

We deduce our main results concerning Question 2.

\begin{theorem}
  \label{theoremODO2}
  Let $L$ be a linear ODO of order $4$.
  Then the degree of the Wronski map $\Wr_{L, 2}$ exceeds $1$ if and only if $L$ is equivalent to a self-adjoint linear ODO.
\end{theorem}

\begin{theorem}
 \label{theoremODO3}
  Let $L$ be a linear ODO of order $6$.
  Then the following statements hold.
\begin{enumerate}
\item  If $\dim X_L^\perp\leq 5$, then the degree of the Wronski map $\Wr_{L, 3}$  is equal to $1$.

\item If $\dim X_L^\perp =6$, then  the degree of the Wronski map  $\Wr_{L, 3}$ exceeds $1$ if and only if $L$  is
  equivalent to a self-adjoint linear ODO. 
\end{enumerate}
\end{theorem}

In the next section, we discuss an application of Theorems~\ref{theoremODO2} and~\ref{theoremODO3} to pole placement in
linear systems theory.
We prove our main results in Section~\ref{S:GWrM}.

\section{Application to Pole Placement for Constant Output Feedback}\label{contrsec}
For a background on linear systems theory, see~\cite{Delch}.
A \demph{state-space realization} of a (strictly proper) $m$-input $p$-output linear system is a triple
$\Sigma=(A,B,C)$ of matrices of sizes $N\times N$, $N\times m$, and $p\times N$.
This defines a system of first order constant coefficient linear differential equations,
 \begin{equation}\label{Eq:State_Space}
   \dot{x}\ =\ Ax + Bu\qquad\mbox{and}\qquad y\ =\ Cx\,,
 \end{equation}
 where $x\in \CC^N$, $u\in \CC^m$, and $y\in \CC^p$ are functions of $t\in\CC$ (and
 $\defcolor{\dot{x}}=\frac{d}{dt}x$).
 Applying Laplace transform ($u(t)\mapsto\hat{u}(s)$) and assuming that $x(0)=0$, we eliminate $\widehat x$ to obtain
\[
   \widehat{y}(s)\ = \ C(sI-A)^{-1}B\, \widehat{u}(s)
    \ =\ G(s)\, \widehat{u}(s)\,,
\]
 where  $\defcolor{G(s)}:=C(sI-A)^{-1}B$ is the \demph{transfer function} of~\eqref{Eq:State_Space}.
 This $p\times m$ matrix of rational functions has poles at the eigenvalues of $A$.

 A linear system may be controlled with output feedback, setting
 $u=Ky$, where $K$ is a constant $m\times p$ matrix.
 Substitution in~\eqref{Eq:State_Space} and elimination gives the closed loop system,
\[
   \dot{x}\ =\ (A + BKC) x\,,
\]
 whose transfer function has poles at the zeroes of the characteristic polynomial
 \begin{equation}\label{Eq:PPP}
   \defcolor{P_\Sigma}(K)\ =\ P_\Sigma\ :=\
    \det( sI \ -\ (A+BKC))\,.
 \end{equation}

 The map $K\mapsto P_\Sigma(K)$ is called the \demph{pole placement map}.
 Given a system~\eqref{Eq:State_Space} with state-space realization $\Sigma$ and poles
 $\defcolor{z}=\{z_1,\dotsc,z_N\}\subset\CC$, the pole placement problem asks for a matrix $K$ such that
 $P_\Sigma(K)$ vanishes at the points of $z$.
 This is only possible for general $z$ if $N\leq mp$~\cite{Byrnes}.
 We are interested when $N\geq mp$ and the pole placement map is a nontrivial branched cover of its
 image.

Using the injection $\Mat_{m\times p}\CC\to \Gr_p\CC^{m+p}$ where $K$ is sent to the
column space of the matrix $(\begin{smallmatrix}K\\I_p\end{smallmatrix})$, standard manipulations
show that the pole placement map is a linear projection of $\Gr_p\CC^{m+p}$.
The map that sends $s\in\PP^1$ to the column space of
$(\begin{smallmatrix}I_m\\ G(s)\end{smallmatrix})$
defines the \demph{Hermann-Martin curve} $\defcolor{\gamma_\Sigma}\colon\PP^1\to\Gr_m\CC^{m+p}$~\cite{HM78}.
Its degree is the McMillan degree, which is the minimal number $N$ in a state-space
realization giving the transfer function $G(s)$.
Such a minimal representation is observable and controllable~\cite{Delch}.

If $X_\Sigma\subset\bwedge^m\CC^{m+p}$ is the linear span of the image of the curve $\gamma_\Sigma$ 
and $\defcolor{Z}:=X_\Sigma^\perp$ is its annihilator in $\bwedge^p\CC^{m+p}$, then the
pole placement map $P_\Sigma$ is the linear projection $\pi_{Z}$, and we may identify the quotient
$X_\Sigma ^*=(\bwedge^p\CC^{m+p})/Z$ with the space of polynomials of degree at most $N$.
The pole placement map is \demph{proper} if $\emptyset\neq\PP Z$ is disjoint from the Grassmannian $\Gr_p\CC^{m+p}$.
This terminology is not standard in systems theory.

Consider the following change of coordinates in the state, input, and output spaces
\begin{equation}
\label{statefeed1}
                 x\ =\ R\widetilde{x}\,,\qquad 
                 u\ =\ Q\widetilde{y}+W\widetilde{u}\,,\quad\mbox{and}\quad
                 y\ =\ T\widetilde{y}\,,
\end{equation}
where $R$, $W$, and $T$ are invertible matrices  and $Q$ is a $m\times p$ matrix.
The transformation of the space $\CC^N\times \CC^m\times \CC^p$ given by \eqref{statefeed1} is a
\demph{state-feedback transformation}.
Substituting \eqref{statefeed1} into \eqref{Eq:State_Space}, we obtain a new state-space realization in
$(\widetilde x, \widetilde u,\widetilde y)$,
 \[
   \dot{\widetilde{x}}\ =\ \widetilde{A}\widetilde{x}+\widetilde{B}\widetilde{u}
   \qquad\mbox{and}\qquad
  \widetilde{y}\ =\ \widetilde C\widetilde{x}\,,
 \]
given by the triple of matrices $\widetilde\Sigma=(\widetilde A,\widetilde B, \widetilde C)$, where
\begin{equation}
\label{statefeedbakmat}
  \widetilde A\  =\ R^{-1}(A+BQT^{-1}C)R\,, \quad \widetilde B\ =\ R^{-1}BW\,,
  \quad\mbox{and}\quad \widetilde C\ =\ T^{-1}CR\,.
\end{equation}
Two realizations are \demph{state-feedback equivalent} if one is a state-feedback transformation of the other.
The following is standard.

\begin{proposition}
 Equivalent state-space realizations have equivalent Hermann-Martin curves, where the equivalence is
 induced by an element of $\GL(\CC^{m+p})$.
\end{proposition}

A state-space realization \eqref{Eq:State_Space} is \demph{symmetric}~\cite{Fuhr} if $A^T=A$ and $C=B^T$. 

\begin{proposition}\cite[Sect.\ 3.2]{HSZ2017}
 \label{symmcontr}
  For a controllable and observable linear system with state-space realization $\Sigma$~\eqref{Eq:State_Space},
  the corresponding center $Z$ is self-adjoint if and only if the realization $\Sigma$ is state-feedback equivalent to
  a symmetric realization. 
\end{proposition}

The degree of the pole placement map of a symmetric state-space realization is at least $2$, because 
$P_\Sigma(K^T)=P_\Sigma(K)$.
The following corollaries are consequences of Theorem \ref{m=2iff},
of Corollary \ref{wendy2}, and of Theorem \ref{mainm=3theor}.

\begin{corollary}
  \label{maincorcontr2}
   If a controllable and observable linear system with $m=p=2$ has a proper pole placement map, then any state-space
   realization \eqref{Eq:State_Space} is state-feedback equivalent to a symmetric realization. 
\end{corollary}

\begin{corollary}
  \label{maincorcontr3}
  Suppose that $\Sigma$ is a state-space realization \eqref{Eq:State_Space} of a controllable and observable linear system
  with $m=p=3$ whose pole placement map is proper and has degree greater than $1$.
  If the center $Z$ of the pole placement map has dimension at most six, then $\dim Z=6$, and $\Sigma$ is
  state-feedback equivalent to a symmetric realization.  
\end{corollary}

\section{Linear projections of the Grassmannian}\label{S:GWrM}

For a finite-dimensional vector space $W$, let \defcolor{$W^*$} be its linear dual.
Write \defcolor{$\PP W$} for its projective space of one-dimensional linear subspaces.
Then \defcolor{$\PP W^*$} is identified with the set of hyperplanes in $W$.
For a vector subspace $Z\subset W$, $\PP Z$ is a linear subspace of $\PP W$.
We will often write $Z$ for $\PP Z$, and $\alpha$ for a nonzero vector in $W$, for the linear subspace
$\langle\alpha\rangle$, and for the corresponding point of $\PP W$.
Context will determine which we intend.

Let $m,n$ be positive integers with $m<n$ and let \defcolor{$V$} be an $n$-dimensional complex vector space.
For a proper linear subspace $Z\subsetneq\PP\bwedge^m V$, the \demph{projection} with 
\demph{center} $Z$, 
 \begin{equation}\label{Eq:projection}
  \PP \bwedge^m V\smallsetminus Z\ \longrightarrow\ \PP\left(\bwedge^m V\right)/Z\,,
 \end{equation}
is induced by the quotient map $\bwedge^m V\twoheadrightarrow (\bwedge^m V)/Z$.
This projection is a rational map on $\PP{\bwedge^m}V$ as it is not defined on $Z$.

The Grassmannian \defcolor{$\Gr_m V$} of $m$-dimensional subspaces of $V$ is embedded into $\PP{\bwedge^m}V$ via the
\demph{Pl\"{u}cker embedding} which sends an $m$-dimensional space $\Lambda$ with basis $v_1,\dotsc, v_m$ to the
span of its Pl\"ucker vector $v_1\wedge\cdots\wedge v_m$, written $\Lambda$.
Elements of $\bwedge^mV$ representing points of $\Gr_m V$ are \demph{decomposable}.
Whether we intend $\Lambda\in\Gr_mV$ to be a point of $\PP{\bwedge^m}V$ or a linear subspace of $V$ will often be
determined by context.

Let $Z\subset\PP{\bwedge^m}V$ be a linear subspace disjoint from  $\Gr_m V$.
Write \defcolor{$\pi_Z$} for the restriction of the corresponding linear projection~\eqref{Eq:projection} to $\Gr_m V$.
In~\cite{HSZ2017} such a linear projection was called a {\it generalized Wronski map}, a terminology 
motivated by the following result.

\begin{proposition}[{\cite[Prop.\ 2.3]{HSZ2017}}]
 \label{equiv}
 The Wronski map $\Wr_{L,m}$ of an $n$th order linear ODO $L$ is the projection $\pi_Z$ with
 center $Z=X_L^\perp$, where $X_L$ is defined by \eqref{defspaceK}.
\end{proposition}

\begin{remark} 
\label{doesnotmeet}
Note that  $X_L^\perp$ is disjoint from the Grassmannian $\Gr_mV_L$.
This is because Wronskians are not identically zero and  the formulation \eqref{ceq}.\hfill$\diamond$
\end{remark}

Assume that $\dim V=2m$.
A 2-form $\sigma\in\bwedge^2V$ is an element of the tensor space $V\otimes V$.
It is a linear map $V^*\to V$ which is given by contraction, $v\mapsto v\lrcorner\,\sigma$.
The \demph{rank} of $\sigma$ is its rank as a linear map, and this is an even integer.
When $\sigma$ has rank $2m$, it is a symplectic form on $V^*$.
Then corresponding map $V^*\to V$ is an isomorphism and $\sigma$ induces a symplectic form
$\defcolor{\sigma^*}\in\bwedge^2 V^*$  on $V$.
The \demph{skew-orthogonal complement} to $\Lambda\in\Gr_mV$ is the linear subspace
\[
  \defcolor{\Lambda^\angle}\ :=\ \{w\in V\mid \sigma^*(w,v)=0\quad\forall v\in\Lambda\}\,.
\]
This also has dimension $m$, so $\Lambda^\angle\in\Gr_mV$.

%
%
A linear subspace $Z\subset\PP{\bwedge^m}V$ 
is \demph{self-adjoint} if there exists a symplectic form $\sigma$ on $V^*$ such that 
 \begin{equation}
  \label{sup2}
   Z\ \supseteq\  \PP\bigl(\CC \sigma\wedge \bwedge^{m-2} V\bigr)\,.
 \end{equation}

By \cite[Cor.\ 1.5]{HSZ2017}, 
 \begin{equation}
  \label{mainidgen}
    \pi_Z (\Lambda^\angle)\ =\ \pi_Z(\Lambda)\,, \quad \forall \Lambda \in \Gr_m V\,,
 \end{equation}
Thus when $Z$ is self-adjoint, the degree of $\pi_Z$ is even and hence exceeds $1$.
We address the converse:
{\it Does degree of $\pi_Z$ exceeding $1$ imply that the center $Z$ is self-adjoint?}  

\subsection{Projection from a point} 
\label{smalldimsec}
Let \defcolor{$\pi_\omega$} be the linear projection with center $\omega\in\PP{\bwedge^m}V$.

\begin{lemma}
\label{dim1lem}
 Suppose that $Z\subset\PP{\bwedge^m}V$ is a linear subspace disjoint from the Grassmannian $\Gr_m V$.
 For $\Lambda,\Lambda'\in \Gr_m V$, we have $\pi_Z(\Lambda)=\pi_Z(\Lambda')$ if and only if there exists a point
 $\omega\in Z$ such that $\pi_{\omega}(\Lambda)=\pi_{\omega}(\Lambda')$ if and only if
 $Z$ meets the line $\langle\Lambda,\Lambda'\rangle$ in $\PP{\bwedge^m}V$ containing the points $\Lambda,\Lambda'$.
\end{lemma}

\begin{proof}
  If  $\pi_{\omega}(\Lambda)=\pi_{\omega}(\Lambda')$, then for any subspace $Z$ containing $\omega$,
  $\pi_Z(\Lambda)=\pi_Z(\Lambda')$.
  For the other direction, suppose that $\pi_Z(\Lambda)=\pi_Z(\Lambda')$ with $\Lambda\neq\Lambda'$ in $\Gr_mV$.
  Then the line $\langle\Lambda,\Lambda'\rangle$ they span meets $Z$.
  If $\omega\in\langle\Lambda,\Lambda'\rangle\cap Z$, then $\pi_{\omega}(\Lambda)=\pi_{\omega}(\Lambda')$.
\end{proof}

For a center $Z\subset\PP{\bwedge^m}V$ disjoint from the Grassmannian $\Gr_mV$, define
 \begin{equation}
  \label{SR}
   \defcolor{\calS_Z}\ :=\ 
   \{\Lambda\in\Gr_m(V) \mid \exists \Lambda'\neq\Lambda \text{ such that } \pi_Z(\Lambda)=\pi_Z(\Lambda')\}\,,
 \end{equation}
and for $\omega\in\PP{\bwedge^m}V$, similarly define $\defcolor{\calS_\omega}$.
Lemma~\ref{dim1lem} is equivalent to 
 \begin{equation}
  \label{SR1}
   \calS_Z\ =\ \bigcup_{\omega \in Z} \calS_{\omega}\,.
 \end{equation}
%

\begin{remark}
\label{schemerem}
 Lemma~\ref{dim1lem} motivates our approach to study the degree of the map $\pi_Z$.
 First, for each $\omega\in Z$, describe all $\Lambda\in \Gr_mV$ such that there exist $\Lambda'\neq\Lambda$ in $\Gr_mV$
 with $\pi_\omega(\Lambda) = \pi_\omega(\Lambda')$.
 Then take a union of all such $\Lambda$ for $\omega\in Z$.
 If this union does not contain an open dense set of $\Gr_mV$ then $\pi_Z$  has degree $1$. 
 
 The group $\GL(V)$ of invertible linear transformations on $V$ acts on $\Gr_mV$ and $\PP{\bwedge^m}V$, and for
 $\omega\in\PP{\bwedge^m}V$, $\Lambda,\Lambda'\in\Gr_mV$, and $g\in\GL(V)$, we have
\[
   \pi_{\omega}(\Lambda)=\pi_{\omega}(\Lambda')
   \qquad\mbox{if and only if}\qquad
   \pi_{g.\omega}(g.\Lambda)=\pi_{g.\omega}(g.\Lambda')\,.
\]
 Therefore, to find the set of pairs  $\Lambda,\Lambda'\in\Gr_m(V)$
 with the same image under $\pi_{\omega}$ it is enough to find this set for one representative of the $\GL(V)$-orbit of
 $\omega$.\hfill$\diamond$
\end{remark}

\begin{remark}\label{Rem:G24}
  Suppose that $\dim V=4$.
  The Grassmannian $\Gr_2V\subset\PP{\bwedge^2}V\simeq\PP^5$ is a quadratic hypersurface.
  Thus, if $\omega\in\PP\bwedge^2V\smallsetminus\Gr_2V$, then
  $\pi_\omega\colon\Gr_2V\to\PP({\bwedge^2}V)/\omega\simeq\PP^4$ has degree two.
  In particular, $\calS_\omega\subset\Gr_2V$ is dense and therefore has dimension four.

  This will be relevant in Section~\ref{Sec:centres}, where we show that for
  $\omega\in\PP{\bwedge^3}\CC^6\smallsetminus\Gr_3\CC^6$, either $\calS_\omega$ is either zero-dimensional, empty, or
  four-dimensional, and the last case may be understood to be a consequence of the projection map on $\Gr_2V$.
  \hfill$\diamond$
\end{remark}

This degree two projection $\Gr_2V\to\PP^4$ is intrinsically related to symplectic structures.

\begin{theorem}
  \label{m=2iff}
  When $\dim V=4$, any $\sigma\in\PP{\bwedge^2}V\smallsetminus\Gr_2V$ is a symplectic form on $V^*$.
  For $\Lambda,\Lambda'\in\Gr_2V$ with $\Lambda\neq\Lambda'$, we have that 
  $\pi_\sigma(\Lambda)=\pi_\sigma(\Lambda')$ if and only if $\Lambda'=\Lambda^\angle$, 
  the skew-orthogonal complement of $\Lambda$ with respect to the symplectic form $\sigma^*$.
\end{theorem}

%
%

\subsection{Projection from a point when \texorpdfstring{$m=3$}{m=3} and \texorpdfstring{$n=6$}{n=6}}
\label{Sec:centres}
Assume that $\dim V=6$.
When convenient, we identify $V$ with $\CC^6$ with the standard basis $\{e_1,\dotsc,e_6\}$ and let
$\{e_1^*,\dotsc,e_6^*\}$ be the dual basis for $V^*$.
Following Remark \ref{schemerem}, we first study the action of $\GL(V)$ on $\PP{\bwedge^3}V$.
The orbits under this action were described by Segre in $1918$~\cite{segre}.
For $i,j,k$, write \defcolor{$e_{ijk}$} for $e_i\wedge e_j\wedge e_k$ and
\defcolor{$e_{ij}$} for $e_i\wedge e_j$.
Then $e_{123}$ is the Pl\"ucker vector of $\langle e_1,e_2,e_3\rangle$.

\begin{theorem} [Segre \cite{segre}, see also \cite{donagi}]
\label{segrethm}
 The action of $\GL(V)$  on $\PP{\bwedge^3}V$ has four orbits $O_0, O_1, O_5$, and $O_{10}$, where $O_i$ has codimension
 $i$. 
 A normal form for an element $\omega_i\in O_i$ of each orbit is as follows.
\begin{enumerate}
\item $\omega_0=e_{123}+e_{456}$,  a point on the line between $e_{123}$ and $e_{456}$.
  
\item $\omega_1=e_{126}-e_{153}+e_{234}$, a general point in the tangent space to  $\Gr_3V$ at $e_{123}$.
  %
  %
  
\item $\omega_5=e_1\wedge(e_{23}+e_{45})$, a point on the line between  $e_{123}$ and $e_{145}$.
  
\item $\omega_{10}=e_{123}$, a point on the Grassmannian $\Gr_3V$. 
\end{enumerate}
\end{theorem}

\begin{remark}
  \label{tangentrem}
  For a 3-plane $\Lambda\in\Gr_3V$, the tangent space $T_\Lambda \Gr_3V$ to the Grassmannian is
  $\mathrm{Hom}(\Lambda, V/\Lambda)$.
  A general point of $T_\Lambda \Gr_3V$ corresponds to an isomorphism $\Lambda\xrightarrow{\sim}V/\Lambda$.
  The normal form in  Theorem~\ref{segrethm}(2) is the point of
  $T_{e_{123}}\Gr_3V$ corresponding to the isomorphism that sends $e_i$ to $e_{i+3}\mod\langle e_1,e_2,e_3\rangle$.
  It is the tangent vector at $t=0$ to the curve 
 \begin{equation}
 \label{123curve}
    \Lambda(t)\ =\ e_1(t)\wedge e_2(t)\wedge e_3(t)\,,
 \end{equation}
 where $e_i(t)=e_i+te_{i+3}$ for $i=1,2,3$.\hfill$\diamond$
\end{remark}

\begin{remark}
\label{orbgeom}
 The tangent variety \defcolor{$\mathcal T X$} of a projective variety $X\subset\PP^N$ is the union of all lines tangent to
 $X$. 
 The orbits from Theorem \ref{segrethm} are described geometrically as follows.
 \begin{enumerate}
 \item
 The orbit $O_0$ is the complement of the tangent variety  $\mathcal T \Gr_3V$ of $\Gr_3V\subset \PP{\bwedge^3}V$.\smallskip
 
 \item
   Let $\mathcal T_1$ be the union of all lines in $\PP{\bwedge^3}V$ connecting two points
   in $\Gr_3V$ whose corresponding subspaces in $V$ have nonzero intersection.
   Then
   \[
     \Gr_3V\ \subset\  \mathcal T_1\ \subset\ \mathcal T \Gr_3V\,,
   \] 
   and  $O_1$ is the complement of  $\mathcal T_1$ in $\mathcal T\Gr_3V$.\smallskip

 \item
 The orbit $O_5$ is the complement of  $\Gr_3V$ in $\mathcal T_1$.

  \item  The orbit $O_{10}$ is $\Gr_3V$.\hfill$\diamond$
\end{enumerate}
\end{remark}

We describe $\calS_\omega$ for $\omega\in\PP{\bwedge^3}V\smallsetminus\Gr_3V$.

\begin{proposition}
 \label{orbitpairprop}
 Let $\omega\in\PP{\bwedge^3}V\smallsetminus\Gr_3V$.
 Then 
 \begin{enumerate}
 \item If $\omega\in O_0$, then $\calS_\omega$ is finite and $\dim \calS_\omega=0$
 \item If $\omega\in O_1$, then $\pi_\omega$ is injective, so that $\calS_\omega=\emptyset$.
 \item If $\omega\in O_5$, then $\dim \calS_\omega=4$. 
\end{enumerate}
\end{proposition}

\begin{proof}
  By Lemma~\ref{dim1lem}, $\Lambda\in\calS_\omega$ if and only if there is a $\Lambda'\in\Gr_3V$ with $\Lambda\neq\Lambda'$
  such that $\omega\in\langle\Lambda,\Lambda'\rangle$, the line in $\PP{\bwedge^2}V$ spanned by the Pl\"ucker vectors of
  $\Lambda$ and $\Lambda'$.

  Let $\Lambda\neq\Lambda'$ be distinct 3-planes in $\Gr_3V$ and
  $\omega\in\langle\Lambda,\Lambda'\rangle\smallsetminus\Gr_3V$.
  By Remark~\ref{orbgeom}(2), $\omega\not\in O_1$, which proves (2).
  We argue by the dimension of $\Lambda\cap\Lambda'$.
  If $\dim\Lambda\cap\Lambda'=0$, then $\omega\in O_0$, by Theorem~\ref{segrethm}(1).
  Since $\dim\Gr_3V=9$ and $\dim\PP{\bwedge^3}V=19$, dimension-counting shows that for a point $\omega\in O_0$,
  $\calS_\omega$  is zero-dimensional and hence finite, proving (1).
  If $\dim\Lambda\cap\Lambda'=1$, then $\omega\in O_5$, by Theorem~\ref{segrethm}(3).
  Statement (3) is Lemma~\ref{L:O5} below.
  If $\dim\Lambda\cap\Lambda'=2$, then $\langle\Lambda,\Lambda'\rangle\subset\Gr_3V$.
\end{proof}

An element $\omega\in\bwedge^3V$ defines two linear maps
 \[
   \begin{array}{rclcrcl}
     \defcolor{\wedge\omega}&\colon& V\ \longrightarrow\ \bwedge^4V &\qquad&
     \defcolor{\lrcorner\,\omega}&\colon& V^*\ \longrightarrow\ \bwedge^2V \\
     && v\ \longmapsto\ v\wedge\omega&&&& v\ \longmapsto\  v\lrcorner\,\omega
   \end{array}\ .
 \]
%

\begin{lemma}\label{L:Kernels} 
  If $\omega\in O_5$, then both $\wedge\omega$ and $\lrcorner\,\omega$ have one-dimensional kernels.
\end{lemma}
\begin{proof}
  Computations using the normal form of $\omega \in O_5$ given by Theorem~\ref{segrethm}(3) show that
  the kernel of $\wedge\omega_5$ is $\langle e_1\rangle$ and the kernel of $\lrcorner\,\omega_5$ is $\langle e_6^*\rangle$.
\end{proof}

For $\omega \in O_5$,
write $\defcolor{\alpha_\omega}\in\PP V$ for the kernel of $\wedge\omega$ and $\defcolor{A_\omega}\in\PP V^*$ for the
kernel of $\lrcorner\,\omega$.
We regard $\alpha_\omega$ as a 1-dimensional linear subspace of $V$ and $A_\omega$ as a hyperplane in $V$.

\begin{corollary}\label{C:IntSpan}
  Let $\omega\in O_5$.
  Then $\alpha_\omega\subset A_\omega$, $A_\omega$ is the smallest subspace $W$ of $V$ such that $\omega\in\bwedge^3W$, and
  if $\pi_\omega(\Lambda)=\pi_\omega(\Lambda')$ for $\Lambda\neq\Lambda'\in\Gr_3V$, then
  $\alpha_\omega=\Lambda\cap\Lambda'$ and $A_\omega=\langle\Lambda,\Lambda'\rangle$ (their span in $V$).
  Finally, there is an indecomposable $2$-form $\sigma\in\bwedge^2 A_\omega$ such that
  $\omega=\alpha_\omega\wedge\sigma$, with $\alpha_\omega$ and $\sigma$ well-defined up to scalars. 
\end{corollary}
\begin{proof}
  By the normal form of Theorem~\ref{segrethm}(3) and the proof of Proposition~\ref{orbitpairprop},
  $\alpha_\omega=\Lambda\cap\Lambda'$, so that $\langle\Lambda,\Lambda'\rangle$ is a hyperplane in $V$.
  Since $\omega,\Lambda,\Lambda'$ are collinear in $\PP{\bwedge^3}V$, $\omega\in\bwedge^3 A_\omega$.
  For any four dimensional subspace $W$ of $V$, $\bwedge^3W\subset\Gr_3V$, which shows the minimality of $A_\omega$.
  The last statement follows from these identifications and Theorem~\ref{segrethm}(3).
\end{proof}

By Corollary~\ref{C:IntSpan}, if $\omega\in O_5$, then
$\omega\in \CC\alpha_\omega\wedge\bwedge^2 A_\omega\simeq \bwedge^2(A_\omega/\alpha_\omega)$.
Notice that $\Lambda \mapsto \Lambda/\alpha_\omega$ identifies the Schubert variety
 \begin{equation}\label{Eq:Xomega}
   \defcolor{\Omega_\omega}\ :=\ \{\Lambda\in\Gr_3V\mid \alpha_\omega\in\Lambda\subset A_\omega\}
 \end{equation}
with $\Gr_2(A_\omega/\alpha_\omega)\simeq \Gr_2\CC^4$.

Let $\defcolor{\Fl}\subset\PP V\times\PP V^*$ be the flag variety whose points are pairs $(\alpha,A)$ with
$\alpha\subset A$; the one-dimensional linear subspace $\alpha$ lies in the hyperplane $A$.
The projection of $\Fl$ to each projective space factor is a $\PP^4$ bundle.
Let $\defcolor{L}\to\Fl$ 
be the subbundle of $\PP{\bwedge^3}V\times\Fl$ whose fiber over $(\alpha,A)$ is
$\PP(\alpha\wedge\bwedge^2A)\simeq\PP^5$.
The Schubert variety $\Omega_\omega$~\eqref{Eq:Xomega} depends only upon the flag $\alpha_\omega\subset A_\omega$ and it
lies in $\PP(\alpha_\omega\wedge\bwedge^2A_\omega)$.
Write \defcolor{$\Omega(\alpha,A)$} for the Schubert variety corresponding to the flag $\alpha\subset A$.
A consequence of this definition and Corollary~\ref{C:IntSpan} is the following.

\begin{corollary}\label{C:Bundle}
  For $\omega\in O_5$, the map $\omega\mapsto(\alpha_\omega,A_\omega)\in\Fl$ realizes $O_5$ as a bundle over $\Fl$, which
  is a dense open subset of $L$. 
  The points in the fiber above $(\alpha,A)$ consist of points in $\PP(\alpha\wedge\bwedge^2A)$ in the complement
  of\/ $\Omega(\alpha,A)$.   
\end{corollary}  

\begin{lemma}\label{L:O5}
  For $\omega\in O_5$, $\calS_\omega$ is a dense subset of\/ $\Omega_\omega$ and therefore has dimension four.
\end{lemma}

\begin{proof}
  In the proof of Corollary~\ref{C:IntSpan}, we observed that if $\Lambda\neq\Lambda'$ are 3-planes in $\Gr_3V$ with
  $\pi_\omega(\Lambda)=\pi_\omega(\Lambda')$, then
  $\alpha_\omega\subset\Lambda\subset A_\omega$.
  This implies that $\calS_\omega\subset \Omega_\omega$.

  Consider the restriction of $\pi_\omega$ to $\Omega_\omega\subset\Gr_3V$.
  Both $\omega$ and $\Omega_\omega$ lie in $\PP(\alpha_\omega\wedge \bwedge^2A_\omega)$, which is identified
  with $\PP{\bwedge^2}(A_\omega/\alpha_\omega)$.
  Write $\omega=\alpha_\omega\wedge\sigma$ with $\sigma\in\bwedge^2 (A_\omega/\alpha_\omega)$.
  Identifying $\Omega_\omega$ with $\Gr_2(A_\omega/\alpha_\omega)$, the map $\pi_\omega$ on $\Omega_\omega$ becomes 
  $\pi_\sigma$, which has degree 2, by Remark~\ref{Rem:G24}.
  This completes the proof.
\end{proof}  

Theorem \ref{m=2iff} and the proof of  Lemma \ref{L:O5} imply the following corollary.

\begin{corollary}
  Let  $\omega =\alpha\wedge \sigma\in O_5$.
  If $\Lambda\neq \Lambda'$ are $3$-planes then $\pi_\omega(\Lambda)=\pi_\omega(\Lambda')$
  if and only if $\Lambda, \Lambda'\in \Omega_\omega$ and $\Lambda'/\alpha=\big(\Lambda/\alpha\big)^{\angle_\sigma}$.
\end{corollary}

 \subsection{The center has dimension less than five}
 \label{small6sec}

By Proposition \ref{orbitpairprop} and \eqref{SR1}, if $Z\subset\PP{\bwedge^3}V$ is a linear subspace that does not meet
the Grassmannian $\Gr_3V$, then 
 \begin{equation}
  \label{SR2}
  \calS_Z\ =\ \bigcup_{\omega\in Z\cap O_0} \calS_\omega\ \cup\ \bigcup_{\omega\in Z\cap O_5} \calS_\omega\, ,
 \end{equation}
and it follows that
 \begin{equation}
  \label{dimSR}
  \dim \calS_{Z}\ \leq\ \max \{\dim (Z\cap O_0), \dim (Z\cap O_5)+4\}\,.
 \end{equation}
Since $\dim \Gr_3V=9$, the last relation implies the following result.

\begin{theorem}
 \label{wendy1}
 If\/ $Z\subset\PP{\bwedge^3}V$ is a linear subspace that does not meet the Grassmannian $\Gr_3V$, $\dim Z<9$, and
 $\dim Z\cap O_5\leq 4$, then $\pi_Z$ has degree $1$ on $\Gr_3V$.
\end{theorem}
\begin{proof}
  From the assumptions and \eqref{dimSR}, we have that $\dim \calS_Z\leq 8$.
  Thus $\Gr_3V\smallsetminus \calS_Z$ contains a nonempty Zariski open set and therefore $\pi_Z$ has degree $1$.
\end{proof}

\begin{corollary}
\label{wendy2}
 If\/ $Z$ does not meet the Grassmannian $\Gr_3V$ and  $\dim Z\leq 4$, then  $\pi_Z$ has degree $1$.
\end{corollary}

\subsection{Five-dimensional center}
\label{6sec}
Let $Z\subset\PP{\bwedge^3}V$ be a linear subspace such that the following three conditions hold,
\begin{enumerate}
 \item[({\it i})]  $\dim Z=5$,
 \item[({\it ii})] $\dim Z\cap O_5\geq 5$, which together with ({\it i}) is equivalent to $\dim Z\cap O_5= 5$, and
 \item[({\it iii})] $Z$ does not meet the Grassmannian $\Gr_3V$, so that  $Z \subset  O_5$.
\end{enumerate}

We establish the following result.

\begin{theorem}
\label{mainm=3theor}
  If\/ $Z\subset\PP{\bwedge^3}V$ is a linear subspace that does not meet the Grassmannian $\Gr_3V$,
   $\dim Z=5$, and the degree of $\pi_Z$ exceeds $1$, then $Z$ is self-adjoint.
\end{theorem}

The hypotheses imply that $Z \subset  O_5$.
We begin with a lemma about lines in $O_5$.
For this, $\omega_i$, $\sigma_i$, $\rho_i$, $\alpha_i$, $v_i$, $w_i$ for $i=1,2$, and  $v$ are vectors and not points in
projective space.

\begin{lemma}
\label{petro}
  Assume that $\omega_1, \omega_2\in O_5$ and the line they span lies in $O_5$.   
  If  $\omega_i=\alpha_i\wedge\sigma_i$ for $i=1,2$ as in Corollary~$\ref{C:IntSpan}$, then 
  one of the following cases holds.
\begin{enumerate}
\item $\langle\alpha_1\rangle=\langle\alpha_2\rangle$. 
\item $\alpha_1$ and $\alpha_2$ are linearly independent and
  $\langle\sigma_1\rangle\equiv\langle\sigma_2\rangle \mod \langle \alpha_1,\alpha_2\rangle$.
  There is a $2$-form $\sigma\in\bwedge^2 V$ such that, up to a scalar factor,
  $\omega_i=\alpha_i\wedge\sigma$ for $i=1,2$.

\item There exist $v,w_1,w_2,v_1,v_2\in V$ where $\alpha_1,\alpha_2,v,v_1,v_2$ are linearly independent
  with $\langle v,v_1,v_2\rangle=\langle v,v_1,w_1\rangle=\langle v,v_2,w_2\rangle$
  such that
  \[
    \omega_1\ =\ \alpha_1\wedge(\alpha_2 \wedge w_1 + v\wedge v_1)
      \qquad\mbox{and}\qquad
    \omega_2\ =\ \alpha_2\wedge(\alpha_1 \wedge w_2 + v\wedge v_2)\,.
  \]    
\end{enumerate} 
\end{lemma}

\begin{proof}
  Suppose that (1) does not hold, so that $\alpha_1$ and $\alpha_2$ are linearly independent.
  Let us suppose that $\alpha_1=e_1$ and $\alpha_2=e_2$.
  Let $\defcolor{U}:= \langle e_1,e_2\rangle$ and $\defcolor{W}=\langle e_3,\dotsc,e_6\rangle\simeq\CC^4$, which are
  transversal. 
  We express $\sigma_1, \sigma_2$ in terms of $e_2$ and $e_1$ respectively.
  We have
  \begin{equation}
    \label{Eq:specialForm}
    \begin{array}{l}
     \omega_1\ =\ e_1\wedge\sigma_1\ =\ e_1\wedge (e_2\wedge w_1 + \rho_1)\,, \\
     \omega_2\ =\ e_2\wedge\sigma_2\ =\ e_2\wedge(e_1\wedge w_2+\rho_2)\,, 
    \end{array}
  \end{equation}
  where $\defcolor{w_1},\defcolor{w_2}\in W$, and $\defcolor{\rho_1},\defcolor{\rho_2}\in \bwedge^2W$ are the terms in  
  $\sigma_1,\sigma_2$ that do not contain $e_2$ and $e_1$ respectively.
  For $i=1,2$,  since $\sigma_i$ is indecomposable, neither $\rho_i$ nor $w_i\wedge\rho_i$ is zero.

  Let $\lambda,\mu\in\CC$ be nonzero.
  Since $\lambda\omega_1 + \mu\omega_2\in O_5$, it has the form $\alpha\wedge\sigma$, where $0\neq\alpha\in V$  is defined
  up to a scalar by $\alpha\wedge(\lambda\omega_1 + \mu\omega_2)=0$.
  Let us write $\alpha=\defcolor{a}e_1+\defcolor{b}e_2+\defcolor{v}$, where $v\in W$.
  The vector $v$ and the coefficients $a$ and $b$ are functions of $\lambda$ and $\mu$, up to a common scalar, and at least
  one of $a$, $b$, and $v$ is nonzero. 
  We use~\eqref{Eq:specialForm} to rewrite $\alpha\wedge(\lambda\omega_1 + \mu\omega_2)=0$ as
  \[
    (ae_1+be_2+v)\wedge (\lambda e_{12}\wedge w_1 +\lambda e_1\wedge\rho_1
                        -\mu e_{12}\wedge w_2 + \mu e_2\wedge\rho_2)\ =\ 0\,.
  \]
  Recall that $e_{12}=e_1\wedge e_2$.
  Expanding gives
  \begin{equation}\label{Eq:messyForm}
    e_{12}\wedge(\mu a\rho_2 - \lambda b \rho_1 + v\wedge(\lambda w_1-\mu w_2))\ -\
    \lambda e_1\wedge v\wedge \rho_1\ -\
    \mu e_2\wedge v\wedge \rho_2\ =\ 0\,.
  \end{equation}
  These summands lie in $e_{12}\wedge\bwedge^2 W$, $e_1\wedge\bwedge^3 W$, and $e_2\wedge\bwedge^3 W$, respectively, and are
  therefore linearly independent.
  This gives the following three equations,
  \begin{eqnarray}
     \mu  a \rho_2 - \lambda b \rho_1  &=& v\wedge ( \mu w_2 - \lambda w_1)\,,  \label{eq:2} \\ 
      v \wedge \rho_1  &=& 0\,, \qquad\mbox{and}\label{eq:3} \\ 
      v \wedge \rho_2 &=& 0\,. \label{eq:4} 
  \end{eqnarray}
  The last two are linear equations for $v\in W$.
  Note that each $\rho_i$ is either decomposable (lies in $\Gr_2 W$) or indecomposable, corresponding to having
  rank 2 or rank 4.
  If either $\rho_1$ or $\rho_2$ is indecomposable and hence of rank 4, then $v=0$ is the only solution.

  Suppose first that $v=0$ is a solution to~\eqref{eq:3} and~\eqref{eq:4}.
  Then~\eqref{eq:2} implies that $\langle\rho_1\rangle=\langle\rho_2\rangle$.
  (We cannot have $ab=0$, for then~\eqref{eq:2} and $(a,b)\neq(0,0)$ implies that one of $\rho_1$ or $\rho_2$ is zero.)
  Scaling $\omega_1$ and $\omega_2$ if necessary, $\rho_1=\rho_2=\rho$, and using~\eqref{Eq:specialForm} 
  we may set $\sigma=e_2\wedge w_1 + e_1\wedge w_2 + \rho$.
  Then Case (2) holds.

  Suppose that~\eqref{eq:3} and~\eqref{eq:4} admit a nonzero solution, $v$.
  Thus $\rho_1$ and $\rho_2$ are each decomposable, and they have the form
  $\rho_i=v_i\wedge v$, for nonzero $v_1,v_2\in W$.
  Then
   \begin{equation}\label{Eq:linCombFormula}
    \lambda \omega_1 + \mu \omega_2\ =\ 
    e_{12}\wedge (\lambda w_1 - \mu w_2) + (\lambda e_1\wedge v_1 + \mu e_2\wedge v_2)\wedge v.
   \end{equation}
  This is indecomposable for $(\lambda,\mu)\neq(0,0)$.

  Suppose that $\langle\rho_1\rangle=\langle\rho_2\rangle$, which corresponds to a 2-plane $\defcolor{H}\subset W$.
  Then~\eqref{eq:2} for all $\lambda,\mu$ implies that $w_1,w_2\in H$.
  In particular, $\rho_1= v'\wedge w_1$, for some $v'\in H$.
  But then $\sigma_1=(e_1+v')\wedge w_1$, which contradicts its being indecomposable.

  Now suppose that $\rho_1$ and $\rho_2$ are linearly independent.
  If $\defcolor{H_i}\in\Gr_2W$ is the 2-plane corresponding to $\rho_i$, then
  $\langle v\rangle=H_1\cap H_2$, and thus $v$ is independent of $\lambda,\mu$ (up to a scalar),
  and we also see that $v,v_1,v_2$ are linearly independent.
  We establish Case (3) by showing that $\langle v,v_1,v_2\rangle=\langle v,v_1,w_1\rangle=\langle v,v_2,w_2\rangle$.

  Consider the 2-forms $\mu  a \rho_2 - \lambda b \rho_1$ for all $\lambda,\mu$.
  If these are all $0$, then $a=b=0$ as  $\rho_1$ and $\rho_2$ are linearly independent.
  Then~\eqref{eq:2} implies that $v,w_1,w_2$ are proportional, which implies that $\sigma_1$ and $\sigma_2$ are
  decomposable, a contradiction.

  Thus, for general $\lambda,\mu$, the 2-form $\mu  a \rho_2 - \lambda b \rho_1\in\bwedge^2\langle v,v_1,v_2\rangle$ is
  nonzero.
  By~\eqref{eq:2}, for all $\lambda,\mu$ we have that $\mu w_2-\lambda w_1\in\langle v,v_1,v_2\rangle$.
  Since $\sigma_1$ is indecomposable, $w_1$ is independent of $v,v_1$, and the same holds for
  $w_2,v,v_2$, which completes the proof.
\end{proof} 

A line in $O_5$ has \defcolor{type $(i)$} if it satisfies condition $(i)$ of Lemma \ref{petro}. 

\begin{corollary}
  \label{moveH}
  Let $\ell\subset O_5$ be a line.
  If $\ell$ has type $(1)$, then $\alpha_\omega$ is the same point in $\PP V$ for every
  $\omega\in\ell$.
  If $\ell$ has type $(3)$, then $A_\omega\in\PP V^*$ is the same hyperplane for every $\omega\in\ell$.
\end{corollary}
\begin{proof}
  The claim about lines of type (1) follows from their definition and Lemma~\ref{petro}(1).
  Suppose $\ell$ has type (3).
  Recall that for $\omega\in O_5$, $A_\omega$ is the unique hyperplane of $V$ with $\omega\in\bwedge^3A_\omega$.
  By the normal form for points on a line of type (3) from Lemma~\ref{petro}(3), we see that
  $A_\omega=\langle \alpha_1, \alpha_2, v, v_1, v_2\rangle$ for all $\omega\in\ell$.
\end{proof}

Now let us define
\begin{equation}
  \label{EZ}
  \begin{array}{rcl}
    E_Z&:=& \{  \alpha_\omega\ \in\ \PP V \mid \mbox{ for\ } \omega\in Z\}\,,\ \mbox{ and}\\
    F_Z&:=& \{ A_\omega\ \in\ \PP V^* \mid \mbox{ for\ } \omega\in Z\}\,.
  \end{array}
\end{equation}

\begin{lemma}\label{L>4}
  If $Z$ is a linear subspace of $\PP{\bwedge^3}V$ of dimension five with $Z\subset O_5$ such that the degree of $\pi_Z$
  exceeds $1$, then $E_Z=\PP V$ and $F_Z=\PP V^*$. 
\end{lemma}

The proof we give uses the following fact about maps between projective spaces.

\begin{proposition}\label{P:Compression}
  If $\phi\colon\PP^r\to\PP^r$ is a nonconstant map, then it is onto.
\end{proposition}
\begin{proof}
  Suppose that $\phi(\PP^r)\neq\PP^r$.
  Since the image is closed, we may compose $\phi$ with the linear projection from a point $x\not\in\phi(\PP^r)$,
  obtaining a map $\psi\colon\PP^r\to\PP^{r-1}$.
  This is given by $r$ homogeneous forms $f_1,\dotsc,f_r$ of the same degree $d$ with no common zeroes;
  for $z\in\PP^r$, $\psi(z)=[f_1(z),\dotsc,f_r(z)]$.
  We must have $d>0$, as $\phi$ and hence $\psi$ is nonconstant.
  This contradicts $f_1,\dotsc,f_r$ having no common zeroes,
  as $r$ forms of degree $d$ define a subvariety in $\PP^r$ of codimension at most $r$.
\end{proof}

\begin{proof}[Proof of Lemma~\ref{L>4}]
  Recall the map $O_5\to\Fl$ that sends $\omega$ to the flag $\alpha_\omega\subset A_\omega$.
  Then $E_Z$ is the image of $Z$ under the further map to  $\PP V$ and $F_Z$ is its image under the map to
  $\PP V^*$. 
  As $Z$, $\PP V$, and $\PP V^*$ are all projective spaces of dimension five, for each of  $\PP V$ and $\PP V^*$,
  the image of $Z$ is either a point, or the map is surjective.

  By Corollary~\ref{C:IntSpan}, if $\Lambda\in\calS_\omega$ for  $\omega\in O_5$, then
  $\alpha_\omega\subset\Lambda\subset A_\omega$.
  If $E_Z$ is a point $\alpha$, then $\calS_Z\subset\{\Lambda\in\Gr_3V\mid \alpha\subset\Lambda\}$, which is a 
  proper subvariety of $\Gr_3V$, and thus $\pi_Z$ has degree 1.
  Similarly, if $F_Z$ is a point, then $\pi_Z$ has degree 1.
\end{proof}

We have another technical lemma.

\begin{lemma} \label{Cartanlemma}
   Given $k{+}1$ linearly independent elements $\{\alpha_i\}_{i=1}^{k+1}$ in $V$, if $\rho\in \bwedge^2V$ satisfies
\begin{equation}
 \label{Cartaneq1}
 \rho\ \equiv\ 0 \mod \langle \alpha_i, \alpha_{k+1}\rangle, \quad \forall i\in\{1,\ldots, k\}\,,
\end{equation}
then up to a nonzero constant, 
\begin{equation}
\label{Cartaneq2}
\rho\ \equiv\ \left\{ \begin{array}{rcl}
          \alpha_1\wedge\alpha_2 \mod\alpha_{k+1}& & k=2\,,\\
          0\mod\alpha_{k+1} & &k>2\,.
        \end{array}\right.
\end{equation}
\end{lemma}
\begin{proof}
 From \eqref{Cartaneq1} it follows that for any $i$ there exist $\beta_i, \gamma_i$ in $V$ such that
 \[
   \rho\ =\ \alpha_i\wedge\beta_i+\alpha_{k+1}\wedge\gamma_i\,.
 \]
 Therefore
 for any $1\leq i \neq j \leq k$
 \begin{equation}\label{Eq:isZero}
   \alpha_{i}\wedge\beta_{i}-\alpha_{j}\wedge\beta_{j}+\alpha_{k+1}\wedge(\gamma_i-\gamma_j)\ =\ 0\,.
 \end{equation}
 Since $\alpha_{i}, \alpha_{j}, \alpha_{k+1}$ are linearly independent, by the classical Cartan lemma we have
 \begin{equation}
   \label{Cartaneq3}
   \beta_{i}\ \in\  \langle \alpha_{i}, \alpha_{j}, \alpha_{k+1}\rangle\,.
 \end{equation}

 If $k>2$, then for any $i\in \{1,\ldots, k\}$, as there is more than one choice of
 $j\in\{1,\ldots, k\}\smallsetminus \{i\}$ in \eqref{Cartaneq3}, we obtain that 
 \begin{equation}
  \label{Cartaneq4}
   \beta_{i}\ \in\  \langle \alpha_{i}, \alpha_{k+1}\rangle\,,
 \end{equation}
 which implies that $\rho \equiv 0\mod \alpha_{k+1}$.

 If $k=2$ then again by~\eqref{Cartaneq3} and~\eqref{Eq:isZero}, we have that 
 \[
   \beta_1\ =\ c\alpha_2 \mod \langle\alpha_1, \alpha_3\rangle\,,
   \qquad
   \beta_2\ =\  -c\alpha_1 \mod  \langle\alpha_2, \alpha_3\rangle\,,
 \]
 for some constant $c$, which completes the proof.
\end{proof}

With these lemmas in place, we give the proof of Theorem~\ref{mainm=3theor}.

\begin{proof}[Proof of Theorem~\ref{mainm=3theor}]
  For this proof, $Z\subset\bwedge^3V$ is a linear subspace of dimension six and $\PP Z$ is its image in $\PP{\bwedge^3}V$.
  By~\eqref{sup2}, to show that $Z$ is self-adjoint, we must produce a form $\sigma\in\bwedge^2V$ such that
  $Z=V\wedge\CC\sigma$.
  
 By Lemma~\ref{L>4}, the maps from $\PP Z$ to each of $\PP V$ and $\PP V^*$ are surjective.
Thus we may choose a basis $\{\omega_i\}_{i=1}^6$ for $Z$ whose images in each of  $\PP V$ and $\PP V^*$ are linearly
independent.
For each $i=1,\dotsc,6$, write $\omega_i=\alpha_i\wedge \sigma_i$, so that $\alpha_i$ is the image of $\omega_i$ in $\PP V$
and let $A_i$ be its image in $\PP V^*$.
Then $\{\alpha_i\mid i=1,\dotsc,6\}$ form a basis for $V$ and  $\{A_i\mid i=1,\dotsc,6\}$ form a basis for $V^*$.
These vectors $\omega_i$, $\sigma_i$, and $\alpha_i$ are only defined up to scalar multiples, so we may freely replace any
by a scalar multiple.

By Corollary~\ref{moveH}, no line $\langle \omega_i,\omega_j\rangle$ for $i\neq j$ has type (1) or (3),
as $\alpha_i$ and $\alpha_j$ are independent and $A_i\neq A_j$.
Therefore, they all have type (2).
By  Lemma~\ref{petro}(2), there exists $\sigma\in \bwedge^2V$ such that $\alpha_i\wedge\sigma_i=\alpha_i\wedge\sigma$ for
$i=1,2$.  
Applying Lemma \ref{petro}(2) to $\langle\omega_1,\omega_3\rangle$ and to $\langle\omega_2,\omega_3\rangle$,
after replacing $\sigma$ and $\sigma_3$ (and possibly $\alpha_1,\alpha_2,\alpha_3$) by scalar multiples,
 \[
   \sigma-\sigma_3\ \equiv\ 0 \ \mod\langle \alpha_i, \alpha_3\rangle\,,\ \mbox{\ for}\quad i=1,2\,.
 \]
By Lemma~\ref{Cartanlemma} for $\rho=\sigma-\sigma_3$ and $k=2$, we have 
\[
  \sigma-\sigma_3 \ \equiv\ c\alpha_1\wedge\alpha_2 \mod \langle\alpha_3\rangle
\]
for some constant $c$.
Consequently, there exists $\beta\in V$ such that 
\[
  \sigma-c\alpha_1\wedge\alpha_2\ =\ \sigma_3+\alpha_3\wedge\beta\,.
\]
Setting $\widetilde{\sigma}:=\sigma-c\alpha_1\wedge\alpha_2$ we get
\begin{equation}
 \label{sigmaremeq3}
 \alpha_i\wedge\sigma_i\ =\ \alpha_i\wedge\widetilde{\sigma}\,, \mbox{\ for} \quad i=1,2,3\,.
\end{equation}

Since the lines between $\omega_4=\alpha_4\wedge\sigma_4$ and $\omega_i$ for $i=1,2,3$ have type (2),
Lemma \ref{petro}(2) implies that after multiplying by scalars, we have
 \begin{equation}
  \label{L6eq5}
   \widetilde{\sigma} \ \equiv\ \sigma_4 \mod \langle \alpha_i,\alpha_4\rangle\,,\mbox{\ for} \quad i=1,2,3\,.
 \end{equation}
 Then, by Lemma~\ref{Cartanlemma} with $\rho=\widetilde{\sigma}-\sigma_4$ and $k=3$ we have
 \[
   \widetilde{\sigma}\ \equiv\ \sigma_4 \mod \alpha_4\,,
 \]
 which implies that in addition to \eqref{sigmaremeq3} we have 
 $\alpha_4\wedge\sigma_4=\alpha_4\wedge\widetilde{\sigma}$.
 The same arguments applied to $\alpha_5$ and $\alpha_6$ imply that  for all $1\leq i\leq 6$,
 we have $\omega_i=\alpha_i\wedge\widetilde{\sigma}$.
 As $\alpha_1,\dotsc,\alpha_6$ form a basis for $V$, we have that
 $Z=V\wedge \CC\widetilde{\sigma}$, which implies that it is self-adjoint,
 and completes the proof of Theorem~\ref{mainm=3theor}.
\end{proof}

\providecommand{\bysame}{\leavevmode\hbox to3em{\hrulefill}\thinspace}
\providecommand{\MR}{\relax\ifhmode\unskip\space\fi MR }
\providecommand{\MRhref}[2]{%
  \href{http://www.ams.org/mathscinet-getitem?mr=#1}{#2}
}
\providecommand{\href}[2]{#2}


\begin{thebibliography}{10}

\bibitem{Byrnes}
C.~I. Byrnes, \emph{Algebraic and geometric aspects of the analysis of feedback
  systems}, Geometrical Methods for the Theory of Linear Systems (C.~I. Byrnes
  and C.~F. Martin, eds.), D. Reidel, Dordrecht, Holland, 1980, pp.~85--124.

\bibitem{Delch}
D.~F. Delchamps, \emph{State space and input-output linear systems},
  Springer-Verlag, New York, 1988.

\bibitem{donagi}
Ron~Y. Donagi, \emph{On the geometry of {G}rassmannians}, Duke Math. J.
  \textbf{44} (1977), no.~4, 795--837.

\bibitem{EH83}
D.~Eisenbud and J.~Harris, \emph{Divisors on general curves and cuspidal
  rational curves}, Invent. Math. \textbf{74} (1983), 371--418.

\bibitem{Fuhr}
P.~A. Fuhrmann, \emph{On symmetric rational transfer functions}, Linear Algebra
  Appl. \textbf{50} (1983), 167--250.

\bibitem{Harris}
J.~Harris, \emph{Algebraic geometry}, Graduate Texts in Mathematics, vol. 133,
  Springer-Verlag, New York, 1992.

\bibitem{HSZ2017}
Yanhe Huang, Frank Sottile, and Igor Zelenko, \emph{Injectivity of generalized
  {W}ronski maps}, Canad. Math. Bull. \textbf{60} (2017), no.~4, 747--761.

\bibitem{HM78}
C.F. Martin and R.~Hermann, \emph{Applications of algebraic geometry to system
  theory: The {McM}illan degree and {K}ronecker indices as topological and
  holomorphic invariants}, SIAM J. Control Optim. \textbf{16} (1978), 743--755.

\bibitem{Sch1886c}
H.~Schubert, \emph{Anzahl-{B}estimmungen f{\"u}r lineare {R}{\"a}ume beliebiger
  {D}imension}, Acta. Math. \textbf{8} (1886), 97--118.

\bibitem{segre}
C.~Segre, \emph{Sui complessi lineari di piani nello spazio a cinque
  dimensioni}, Annali di Mat. pura ed applicata \textbf{27} (1918), no.~1,
  75--123.

\end{thebibliography}
\end{document}